\documentclass[10pt]{article}
\usepackage[affil-it]{authblk}

\usepackage[english]{babel}
\usepackage{amsmath,amssymb,amsthm}
\usepackage{enumerate, fullpage}
\usepackage{graphicx}
\usepackage{color}
\usepackage{transparent}

\newtheorem{theorem}{Theorem}[section]
\newtheorem{proposition}[theorem]{Proposition}
\newtheorem{lemma}[theorem]{Lemma}
\newtheorem{corollary}[theorem]{Corollary}

\theoremstyle{definition}

\newtheorem{remark}[theorem]{Remark}

\newcommand{\R}{\mathbb{R}}

\newcommand{\E}{\operatorname{{\bf E}}}
\newcommand{\PP}{\operatorname{{\bf P}}}
\renewcommand{\H}{\mathbb{H}}
\newcommand{\bd}{\partial}
\newcommand{\dist}{\operatorname{dist}}

\newcommand{\indicate}[1]{\mathbf{1} \left \{ #1 \right \}}
\newcommand{\1}[1]{\indicate{#1}}
\newcommand{\F}{\mathcal{F}}

\newcommand{\haussdim}{\mathrm{dim_H} \,}

\newcommand{\hdim}{\dim_{\operatorname{H}}}

\newcommand{\mN}{\mathcal{N}}

\let \le \leqslant
\let \leq \leqslant

\let \ge \geqslant
\let \geq \geqslant

\newcommand{\ee}{\epsilon}

\begin{document}

\title{A Dimension Spectrum for SLE Boundary Collisions}

\author{Tom Alberts
\thanks{alberts@math.utah.edu}}
 \affil{University of Utah}
\author{Ilia Binder%
\thanks{ilia@math.toronto.edu}}
\affil{University of Toronto}
\author{Fredrik Viklund
\thanks{fredrik.viklund@math.kth.se}}
\affil{KTH Royal Institute of Technology}
\date{}

\maketitle

\begin{abstract}
We consider chordal SLE$_\kappa$ curves for $\kappa > 4$, where the intersection of the curve with the boundary is a random fractal of almost sure Hausdorff dimension $\min\{2-8/\kappa,1\}$. We study the random sets of points at which the curve collides with the real line at a specified ``angle'' and compute an almost sure dimension spectrum describing the metric size of these sets. We work with the forward SLE flow and a key tool in the analysis is Girsanov's theorem, which is used to study events on which moments concentrate. The two-point correlation estimates are proved using the direct method.
\end{abstract}

\section{Introduction}

The Schramm-Loewner evolution curves (SLE$_\kappa$, $\kappa > 0$) appear as scaling limits of interfaces in a number of discrete planar lattice models from statistical mechanics and have become an essential tool in the study of these models.
The geometry of SLE$_\kappa$ is interesting in its own right and depends strongly on the parameter $\kappa > 0$.
If $\kappa > 4$, the curves have non-traversing self-intersections and collide with the boundary of their domain, which is usually taken to be the upper half-plane $\mathbb{H}$. In this paper we study a multifractal spectrum describing the fine geometry of SLE$_\kappa$ curves' boundary collisions when $\kappa >4$.  It is well-known that the points on the boundary of $\H$ that are hit by the curve form a random fractal of dimension $\min\{2-8/\kappa,1\}$ \cite{AS:dimension, SZ:boundary_proximity}. We will prove that this fractal can be decomposed into a family of sets of smaller dimensions according to the ``angle'' at which the curve hits the boundary.

The deterministic case provides some motivation for studying this problem. Consider the half-plane Loewner equation
\begin{align*}
\partial_t g_t(z) = \frac{2}{g_t(z) - U_t}, \quad g_0(z) = z,
\end{align*}
with $U_t$ a continuous real-valued function called the driving term. The family of solutions $(g_t)$ is called a Loewner chain.  If $U_t$ is sufficiently regular then there is a curve $\gamma = \gamma(t)$ that generates this chain, in the sense that $g_t$ is a conformal map from the unbounded connected component $H_t$ of $\H \setminus \gamma[0,t]$ onto $\H$. We write $h_t = g_t - U_t$ for the recentered maps. This gives a correspondence between driving terms and Loewner curves, but translating properties of the driving term into geometric properties of the curve is often difficult.

For some driving terms the corresponding curve hits the real line, and boundary collisions for Loewner curves with deterministic driving terms are studied in \cite{LMR:collisions}. Here the interesting setting is the ``critical'' case of driving terms in the H\"{o}lder-$1/2$ class with seminorm not too small. The model example is the curve driven by $U_t = k \sqrt{1-t}$ with $k > 4$. In this case one can see that the curve stays in the half-plane until time $1$ when it hits the boundary. It turns out that the curve meets the real line at the specific angle
\begin{align}\label{angle-limit}
\lim_{t \to 1} \arg \left( \gamma(t) - \gamma(1) \right) = \pi \frac{1 - \sqrt{1 - 16/k^2}}{1 + \sqrt{1 - 16/k^2}}.
\end{align}
In fact this is true in a certain local sense for curves driven by a H\"{o}lder-$1/2$ function with $\lim_{t \to 1} (1-t)^{-1/2} U_t > 4$, see \cite{LMR:collisions} for these results.

The SLE$_\kappa$ process is defined by taking the Loewner driving term to be $\sqrt{\kappa}$ times a standard Brownian motion. This Brownian driving function is not H\"{o}lder-$1/2$, but for $\kappa > 4$ the corresponding curve almost surely has boundary collisions. We will show that an analogue of the angular behavior described above is still present. Of course, we cannot expect a statement as strong as \eqref{angle-limit} to hold for the rough SLE curve. Instead we will consider a natural generalization involving a decay rate of harmonic measure, analogous to the customary interpretation of the local dimension of harmonic measure as describing a generalized (inverted) angle \cite{makarov_fine}. Suppose $\gamma$ is a Loewner curve as in \eqref{angle-limit} and that it hits $\mathbb{R}$ at $x$. (By this we mean that $\dist(x, \gamma[0, \infty)) = 0$, which is slightly weaker than requiring that there is $t_x$ such that $\lim_{t \uparrow t_x}\gamma(t) = x$.) For $s \ge 0$, let $\tau_s=\tau_s(x)$ be the first time $\gamma$ comes within distance $e^{-s}$ of $x$:
\begin{align*}
\tau_s(x) = \inf \left \{ t \geq 0 : \dist(x, \gamma[0,t]) \leq e^{-s} \right \}.
\end{align*}
Any crosscut of $H_{\tau_s}$ connecting $x$ with $\gamma(\tau_s)$ partitions $H_{\tau_s}$ into exactly two components, one bounded and one unbounded. We let $E_s$ be the part of $\partial H_{\tau_s}$ that is in the boundary of the bounded component, and we consider the (normalized) harmonic measure from infinity of $E_s$
\[
\omega_{\infty}(E_s) := \lim_{y \to \infty} y\omega(iy, E_s, H_{\tau_s}).
\]
For the curves described in \eqref{angle-limit} it is not hard to check that there is a constant $c \in (0,\infty)$ such that as $s \to \infty$,
\[
c^{-1} e^{-\alpha s} \le \omega_{\infty}(E_s) \le c e^{-\alpha s}, \quad \alpha = \frac{\pi}{\pi-\theta },
\]
where $\theta$ is given by the right-hand side of \eqref{angle-limit}. Moreover, which in this case is equivalent, for a different constant $c$,
\[
c^{-1} e^{-\beta s} \le h'_{\tau_s}(x) \le c e^{-\beta s}, \quad \beta = \alpha-1.
\]
For the SLE curves it is too restrictive to consider up-to constants asymptotics, so we shall allow sub-exponential corrections. More precisely, if $(h_t)$ is the (centered) SLE$_\kappa$ forward Loewner chain, then for $\beta \in \R$ we consider the random sets $V_{\beta}$ of points on the line at which the derivative $h_{\tau_s}'(x)$ has a prescribed exponential rate of decay:
\begin{align}\label{eqn:V_b_defn}
V_{\beta} := \left \{ x \in \R_+ : \lim_{s \to \infty} \frac{1}{s} \log h'_{\tau_s}(x) = -\beta, \, \tau_s = \tau_s(x) < \infty \,\, \forall \,\, s > 0 \right \}.
\end{align}
We may interpret this as describing a generalized hitting ``angle'' of the real line. Our main result is:
\begin{theorem}\label{thm:main}
Let $\kappa > 4, a = 2/\kappa,$ and $\beta \in [\beta_-, \hat{\beta}]$ with
\begin{align*}
\beta_{-} := \frac{a}{\frac32 + \sqrt{2-4a} - 2a}, \quad
\hat{\beta} : = \frac{2a}{4a-1}.
\end{align*}
Then almost surely
\begin{align*}
\hdim V_{\beta} = d(\beta) := \frac32 + \beta \left(1 - \frac{1}{8a} \right) - a \frac{(1 + 2\beta)^2}{2 \beta}.
\end{align*}
\end{theorem}

\begin{remark}
Note that  $d(\beta)$ is actually non-negative on a larger interval $\beta \in [\beta_{-}, \beta_{+}]$, where $\beta_{+}$ is given by $a/\left(3/2 - \sqrt{2-4a} - 2a \right)$.
The dimension function $d(\beta)$ achieves its maximal value of $2-4a$ (which is the dimension of $\gamma \cap \R_+$) at $\beta = \hat{\beta}$ (which corresponds to $\lambda = 0$ and $\nu = 4a-1$). To the left of $\hat{\beta}$ the function is increasing and to the right it is decreasing. Our theorem is only stated for $\beta \in [\beta_-, \hat{\beta}]$, which implies $\lambda \ge 0$, but most of our results still hold for the whole range $\beta \in [\beta_{-}, \beta_+]$. In particular, we prove the almost sure upper bound on the dimension of $V_{\beta}$ for this entire range. (Essentially, we also prove a suitably formulated expectation dimension result for the whole range.) However, we only prove the correlation estimate and lower bound on the almost sure dimension for the range of $\lambda \ge 0$ such that $\beta \in [\beta_-, \hat \beta]$. This is similar to the restriction for the tip spectrum in \cite[Section 6]{JL:tip}.
\end{remark}

\begin{remark}Here is a rough idea for the proof of Theorem~\ref{thm:main}. Each $\beta$ corresponds to a particular $\lambda$ with the property that $\E \left[h'_{\tau_s}(x)^\lambda  \1{\tau_s < \infty}\right]$ concentrates on the event that $h'_{\tau_s}(x) \approx e^{-\beta s}$. If $\E \left[h'_{\tau_s}(x)^\lambda \1{\tau_s < \infty} \right] \asymp e^{-\nu s}$ with $\nu=\nu(\lambda)$, then this means one roughly has $\PP\left\{\tau_s < \infty,\, h'_{\tau_s}(x) \approx e^{-\beta s}\right\} \approx e^{-(\nu-\lambda\beta)s}$. This predicts the dimension function $d(\beta) = 1 - (\nu-\lambda \beta)$ and essentially proves it in the sense of expectation dimension.

The proof actually starts by studying the moments of $h'_{\tau_s}$ for a suitable range of $\lambda$. The value of $\beta$ appears by analyzing the behavior in the measure weighted by a martingale given by (a time-change of) $h'_{t}(x)^\lambda$ times a compensator via Girsanov's theorem. A bound on an exponential moment shows that $-\log h'_{\tau_s}(x) = \beta s + O(\sqrt{s})$ with high probability in the weighted measure. The construction of a Frostman measure requires sharp upper bounds on
\[
\E\left[h'_{\tau_s}(x)^\lambda h'_{\sigma_s}(y)^\lambda I_s(x) I_s(y)\right], \quad (\tau_s=\tau_s(x), \, \, \sigma_s=\tau_s(y)).
\]
Here $I_s(x)$ can be thought of as the indicator of the event that $\tau_s<\infty, \, h'_{\tau_s}(x) \approx e^{-\beta s}$. The detailed one-point analysis is crucial for this step, which is performed by looking at the various ways in which the curve can get close to the points, and estimating accordingly.
\end{remark}

With a little bit of work, the main result can also be phrased in terms of decay of harmonic measure. Letting
\[
\Xi_\alpha:=\left\{x>0: x \in \gamma[0, \infty); \, \lim_{s \to \infty} \frac{1}{s} \log \omega_\infty\left(E_s, H_{\tau_s}\right) = -\alpha  \right\},
\]
where $E_s$ is as defined above, the dimension function becomes
\[
\haussdim \, \Xi_\alpha =\frac{1}{2}+\frac{\kappa}{16}+ \alpha\left(1 -\frac{\kappa}{16}\right) -\frac{(1-2 \alpha )^2}{\kappa(\alpha-1)}.
\]
We will comment on this in Section~3.

Besides Theorem~\ref{thm:main} there are two almost sure dimension spectrum results for SLE. The first result appeared in \cite{JL:tip}, which uses the reverse flow to study the behavior of the deriviative close to the tip of the growing SLE curve. That paper builds on work in \cite{Law:mf, JL:optimal_exponent} and the main correlation estimate is proved using similar methods as in Section~3.1 of this paper. The second result is the very recent \cite{miller:mf} which describes the distortion spectrum of the derivative away from the tip. It relies on couplings of SLE and Gaussian free fields and so-called imaginary geometry techniques to estimate correlations. (The paper \cite{BS} precedes these works but computes the integral means spectrum in expectation rather than in the almost sure sense.)

The present paper uses similar ideas as in \cite{Law:mf,JL:tip} but we work directly with the forward flow. In particular, an important tool is Girsanov's theorem which is used to study the events on which moments concentrate. We believe the work here illustrates the method quite well; the proofs, including the two-point estimates, can be made comparatively short and transparent. The correlation estimates are different than those of \cite{Law:mf,JL:tip, miller:mf}, and are proved in a manner more similar to Beffara's argument for the dimension of the curves \cite{Bef}, taking as input the one-point analysis and making use of the domain Markov property. We stress that we work close to the boundary, which simplifies the argument.

A few auxiliary results that appear here have also appeared elsewhere. We mention in particular Lawler's very recent and closely related paper \cite{Law:boundary_mink}. It constructs the Minkowski content of the intersection of the curve with the real line. That paper and this one could be viewed as refining \cite{AS:dimension, AS:measure} in two different directions. Lawler's paper does not explicitly study multifractal properties, but the needed analysis is similar and we have used ideas from its discussion of correlation estimates. It would be interesting to try to strengthen our results in the direction of \cite{Law:boundary_mink} to construct a family of covariant measures on the boundary, supported on the sets we study in this paper. We also mention \cite{miller_wu} which studies a related question using the imaginary geometry technology. There is no significant overlap with this paper but it would be interesting to see if the methods of \cite{miller_wu} can be adapted to prove the correlation estimates of Section~\ref{thm:correlation}.

A few words about notation. We will write $f(x) \asymp g(x)$ when the estimates $c^{-1} g(x) \le f(x) \le c f(x)$ hold, $f(x) \lesssim g(x)$ when $f(x) \le c g(x)$ holds; in all cases the constants do not depend on $x,g,f$ but may depend on the various fixed parameters $a, \lambda$, etc. Constants may change from line to line, as may \emph{subpower functions}, i.e., $\psi$ such that for every $\ee > 0$, $\lim_{x \to \infty} x^{-\ee} \psi(x)=0$. Sub-exponential functions are defined similiarly, requiring $\lim_{x \to \infty} e^{-\ee x} \psi(x)=0$.
\subsection{Acknowledgements}

We thank MSRI and the organizers of the 2012 MSRI program on Random Spatial Processes, where this project was initiated. We also thank Greg Lawler and Wendelin Werner for helpful conversations. Alberts was supported by the Scott Russell Johnson Research Fellowship. Binder was partially supported by an NSERC Discovery grant. Viklund was supported by the Simons Foundation, NSF Grant DMS-1308476, and the Swedish Research Council.

\section{Schramm-Loewner evolution}
\subsection{SLE forward flow}
Throughout we consider the version of the Loewner equation defined by
\begin{align}\label{eqn:Loewner}
\bd_t g_t(z) = \frac{a}{g_t(z) - U_t}, \quad g_0(z) = z,
\end{align}
where $U : [0, \infty) \to \R$ is a continuous function. The SLE$_\kappa$ maps are obtained by setting $a = 2/\kappa$ and $U_t = B_t$, where $B$ is a one-dimensional standard Brownian motion with $B_0 = 0$ and filtration $\F_t$. We let $h_t$ be the shifted map
\begin{align*}
h_t(z) = g_t(z) - B_t,
\end{align*}
so that $h_t$ follows the stochastic differential equation
\begin{align}\label{eqn:h_SDE}
d h_t(z) = \frac{a}{h_t(z)} \, dt - dB_t.
\end{align}
For each $z \in \overline{\H}$ we define the stopping time $T_z$ to be the first time at which $z$ hits zero, i.e.
$T_z = \inf \{ t \geq 0 : h_t(z) = 0 \}$. From this family of stopping times we define the SLE hull $K_t$ by
$K_t = \{ z \in \overline{\H} : T_z \leq t \}$. One of the fundamental properties of the SLE hull is that it is generated by a continuous curve (almost surely). We let $\gamma : [0, \infty) \to \overline{\H}$ denote the corresponding SLE$_\kappa$ curve.

We will work almost exclusively with the curve near $\R$, and since the law of SLE is symmetric about the vertical axis it is sufficient for us to consider $x \in \R_+$. Differentiating \eqref{eqn:h_SDE} with respect to the spatial parameter and solving the resulting ODE yields the formula
\begin{align}\label{eqn:ht_prime_exp}
h_t'(x) = \exp \left \{ - \int_0^t \frac{a}{h_s(x)^2} \, ds \right \}.
\end{align}
Notice that $0 \leq h_t'(x) \leq 1$ for all $x \in \R_+$ and we see that $t \mapsto h_t'(x)$ is decreasing. Also observe that, from Ito's formula, we have the equation
\begin{align} \label{eqn:ht_exp}
h_t(x) = x \exp \left \{ \int_0^t \frac{a - 1/2}{h_s(x)^2} \, ds - \int_0^t \frac{dB_s}{h_s(x)} \right \}.
\end{align}
Clearly both $h_t$ and $h_t'$ are increasing functions of $x$ on $\R_+$.

In studying the distance from a point $x \in \R$ to the curve $\gamma$ it is useful to consider
\begin{align}\label{eqn:Delta_t}
\Delta_t(x) := \frac{h_t(x)}{h_t'(x)} = x \exp \left \{ \int_0^t \frac{2a-1/2}{h_s(x)^2} \, ds - \int_0^t \frac{dB_s}{h_s(x)} \right \},
\end{align}
the last equality following from a combination of \eqref{eqn:ht_prime_exp} and \eqref{eqn:ht_exp}. By the Koebe 1/4 Theorem
\begin{align}\label{eqn:x_Delta_relationship}
\dist(x, \gamma[0,t]) \leq 4 \Delta_t(x)
\end{align}
but there is no \emph{a priori} lower bound in terms of $\Delta$ only; see below.

We define the stopping times
\begin{align}\label{eqn:tau_defn}
\tau_s := \tau_s(x) = \inf \{ t \geq 0 : \dist(x, \gamma[0,t]) \leq e^{-s} \},
\end{align}
with the infimum of the empty set being infinite, as usual. Note these times are non-decreasing in $s$.

\subsection{One-point martingales}\label{sec:conformal_radius}
Suppose $\kappa \in (4,8)$. For each fixed value of $a=2/\kappa \in (1/4, 1/2)$ there is a one-parameter family of ``covariant'' martingales that will be extensively used in our analysis. (The martingales exist for other values of $a$, but we will only consider this case.) We parameterize this family with the variable $\nu$. Let $\nu_0$ and $\nu_c$ be the specific values
\begin{align}
\nu_0 = 4a - 1, \quad \nu_c = \frac{\nu_0}{2} = 2a - \frac{1}{2},
\end{align}
and then define the parameters $\beta$ and $\lambda$ as
\begin{align}\label{eqn:beta_lambda_def}
\beta = \frac{a}{\nu - \nu_c}, \quad \lambda = \frac{\nu}{2a}(\nu - 2\nu_c) = \frac{1}{a} \left( \frac{\nu^2}{2} - \nu_c \nu \right).
\end{align}
We will work only with the case $\nu > \nu_c$; the reason for this choice will be made clear later. Observe that $\beta$ is strictly positive while $\lambda$ is negative for $\nu \in (\nu_c, \nu_0)$ and positive for $\nu > \nu_0$. One can parameterize in other equivalent ways. The various parameters are related by the formulas:
\begin{align*}
0 < \beta < \infty:   \,\, \nu &= \frac{a}{\beta} + \nu_c, \quad \lambda = \frac{1}{2a} \left( \frac{a}{\beta} + \nu_c \right) \left( \frac{a}{\beta} - \nu_c \right) \\
-\frac{\nu_c^2}{2a} < \lambda < \infty: \, \, \nu &= \nu_c + \sqrt{\nu_c^2 + 2a \lambda}, \quad \beta = \frac{a}{\sqrt{\nu_c^2 + 2 a \lambda}}
\end{align*}
Note that $\beta = \nu'(\lambda)$, which will be used later on. For each $x > 0$ we see from \eqref{eqn:ht_prime_exp} and \eqref{eqn:Delta_t} that
\begin{align}\label{eqn:M_def}
M_t^{\nu}(x) := h_t'(x)^{\lambda} \Delta_{t}(x)^{-\nu} = h_t'(x)^{\lambda + \nu} h_t(x)^{-\nu}
\end{align}
is a local martingale on $0 \leq t < T_x$ with
\begin{align}\label{eqn:M_exp_form}
M_t^{\nu}(x) = x^{-\nu} \exp \left \{ \nu \int_0^t \frac{dB_s}{h_s(x)} - \frac{\nu^2}{2} \int_0^t \frac{ds}{h_s(x)^2} \right \}.
\end{align}
Equivalently we have the SDE form
\begin{align*}
\frac{dM_t^{\nu}(x)}{M_t^{\nu}(x)} = \frac{\nu}{h_t(x)} \, dB_t, \quad M_0^{\nu}(x) = x^{-\nu}.
\end{align*}

\begin{remark}\label{rm:u_time}
The representation \eqref{eqn:M_exp_form} suggests using the $x$-dependent random time change $u_x(t) := \int_0^t h_s^{-2} ds$. We will not make much use of this particular time change, but we note that when it is employed many formulas become simple and the ``correct'' behavior can often be read-off easily. Let us write $\hat{h}_u=h_{\hat{t}(u)}$, etc, for the time-changed processes. The process $\hat{M}^{\nu}_u$ is a clearly a martingale. Note that in this time-parameterization $\hat{h}'_u(x) = e^{-au}$. It is straightforward to check that $\hat{\Delta}_u(x) \to 0$ under the probability measure weighted by $\hat{M}^{\nu}_u(x)$, so long as $\nu > \nu_c$. Consequently, since $ \dist(x, \gamma[0,s]) \lesssim \Delta_s(x)$, for each $s < \infty$ we have $\tau_s < \infty$ almost surely under this weighted measure. A similar argument shows that $\tilde{t}(s) < \infty$ almost surely under the measure weighted by the time-changed martingale of Section~\ref{sect:weighted_measure}.
\end{remark}
The estimate \eqref{eqn:x_Delta_relationship} is only one-sided because $h_t(x)$ is larger than the actual distance from the image of $g_t(x)$ to the image set $g_t(\partial K_t)$. To gauge the true distance we need to keep track of the image of the rightmost point on the real line that belongs to the hull at time $t$, that is,
\begin{align}\label{eqn:Ot_def}
O_t := \lim_{x \downarrow \sup K_t \cap \R} h_t(x).
\end{align}
When the tip $\gamma(t)$ is away from $\R_+$ the $O_t$ process evolves as does any other point on the line, that is according to
\begin{align}\label{eqn:Ot_sde}
dO_t = \frac{a}{O_t} \, dt - dB_t
\end{align}
which is the same as \eqref{eqn:h_SDE}. The times at which $\gamma(t) \in \R_+$ correspond to times for which $O_t = 0$. In order to maintain scale invariance for the SLE hull these times must have Lebesgue measure zero. This implies that $O_t$ is a Bessel process with instantaneous reflection to the right at zero; such processes are well-defined in several equivalent ways. See \cite[Chapter 5]{RY:book} for details. We will only use that $O_t$ satisfies the SDE \eqref{eqn:Ot_sde} away from zero and is instantaneously reflecting to the right at zero. The latter implies that the generator of $O_t$ only acts on functions $f : \R_+ \to \R$ satisfying the Neumann condition $f'(0) = 0$.

\newcommand{\crad}{C}

\newcommand{\A}{A}
With $O_t$ at our disposal the Koebe-$1/4$ Theorem gives the two-sided bound
\begin{align}\label{eqn:Koebe_bound}
\frac14 \left( \frac{h_t(x) - O_t}{h_t'(x)}  \right) \leq \dist(x, \gamma[0,t]) \leq 4 \left( \frac{h_t(x) - O_t}{h_t'(x)} \right).
\end{align}
For $x \in H_t$ we let \begin{align} \crad_t(x) = \frac{h_t(x)-O_t}{h'_t(x)}. \label{eqn:Ct_defn} \end{align}This is (a constant times) the conformal radius of the domain $H_t$ reflected about the real axis, as seen from $x$. We will repeatedly use that $\crad_t(x) \asymp \dist(x, \gamma[0,t])$.

We also define a new process $\A_t = \A_t(x)$ by
\begin{align}\label{eqn:A_process}
\A_t(x) = 1 - \frac{O_t}{h_t(x)},
\end{align}
which is a diffusion on $[0,1]$, and we observe that $A_t(x) \Delta_t(x) = \crad_t(x)$. Equation \eqref{eqn:Ot_sde} combined with a simple computation shows that $A_t$ follows the SDE
\begin{align}\label{eqref:A_sde}
d \A_t = \A_t \left[ (1-a) - \frac{a}{1-\A_t} \right] \, \frac{dt}{h_t(x)^2} + \A_t \frac{dB_t}{h_t(x)}.
\end{align}
Observe that $\A_t$ is absorbing at $x = 0$ (since $A_t = 0$ only when $x$ is swallowed by the curve) and instantaneously reflecting at $x = 1$ (by the reflecting property of $O_t$). Using this SDE and $C_t = A_t \Delta_t$ it is straightforward to compute that
\begin{align*}
dC_t = -a C_t \frac{A_t}{1-A_t} \, \frac{dt}{h_t(x)^2},
\end{align*}
so by defining a process $\tilde{t}(s)$ as the solution to the equation
\begin{align*}
s = \int_0^{\tilde{t}(s)} \frac{A_r}{1 - A_r} \, \frac{dr}{h_r(x)^2}
\end{align*}
we have that $\tilde{C}_s := C_{\tilde{t}(s)} = e^{-as}$. We will write $\tilde{h}_s = h_{\tilde{t}(s)}$, $\tilde{M}_s = M_{\tilde{t}(s)}$, etc for the processes time-changed by $\tilde{t}(s)$. All time-changed processes are adapted to the filtration $\tilde{\F}_s = \F_{\tilde{t}(s)}$. This particular time change is often called the \textit{radial parameterization} and also appears in, e.g., \cite{Law:boundary_mink}. Of particular importance for the next section are the local martingales $M_t^{\nu}$, which in terms of the processes $C_t$ and $A_t$ can be written as
\[
M_t^{\nu}(x) = h'_t(x)^\lambda \crad_t(x)^{-\nu} A_t(x)^\nu.
\]
Under the radial parameterization this becomes
\[
\tilde{M}_s^{\nu}(x) = \tilde{h}_s'(x)^{\lambda} e^{\nu a s} \tilde{A}_s(x)^{\nu}.
\]
We remark that when $\lambda > 0$ this is a bounded process on every compact time interval, and therefore a martingale. When $\lambda < 0$ we can use, e.g., the stopping-times discussed in Remark~\ref{rm:u_time}, and take $T_n = \inf\{s \ge  0 : \tilde{h}_s'(x) \le e^{-an} \}$ as a localizing sequence.

\subsubsection{Weighted measure}\label{sect:weighted_measure}
For each $x > 0$ and $\nu > \nu_c$ we use the local martingales $\tilde{M}_s^{\nu}(x)$ to define a new probability measure $\PP^* = \PP^*_{x, \nu}$ by
\[
\PP^*\left( A \right) = \tilde{M}_0^{\nu}(x)^{-1} \E \left[ \tilde{M}_s^{\nu}(x) \operatorname{1}_A \right]
\]
for all $\tilde{\F}_s$-measurable events $A$. As usual this is understood in the sense of localization. This measure then extends to all of $\tilde{\F}_{\infty}$. Using Girsanov's theorem the equation for $\tilde{\A}_s$ under $\PP^*$ is seen to be
\begin{equation}\label{eq:tilde-a}
d \tilde{\A}_s = \left((1-2a + \nu)-(1-a+\nu) \tilde{A}_s \right) ds - \sqrt{\tilde{\A}_s(1-\tilde{\A}_s)} d\tilde{W}_s,
\end{equation}
where $\tilde{W}_s$ is a standard Brownian motion under $\PP^*$. This process has an explicit invariant density which can be determined as follows. Let
\[
\sigma(x)^2=x(1-x), \quad b(x) = (1-2a + \nu)-(1-a+\nu)x,
\]
so that the invariant density $\tilde{\pi} = \tilde{\pi}_{\nu}(x)$ solves the adjoint equation
\[
\frac{1}{2}\left[\sigma(x)^2 \tilde{\pi}(x) \right]'' - \left[b(x) \tilde{\pi}(x) \right]' = 0.
\]
This expands to
\[
x(1-x)\tilde{\pi}''(x) + 2\left( 2a - \nu + (\nu - a - 1) x \right) \tilde{\pi}'(x) + 2 \left( \nu - a \right) \tilde{\pi}(x)=0.
\]
The boundary behavior of $\tilde{A}$ implies that the boundary conditions are
\begin{align*}
\lim_{x \downarrow 0} \sigma^2(x) \tilde{\pi}(x) = 0, \quad \lim_{x \uparrow 1} b(x) \tilde{\pi}(x) - \frac12 (\sigma^2 \tilde{\pi})'(x) = 0.
\end{align*}
By finding the integrable solution to the equation, we conclude the following.
\begin{lemma}\label{lem:dec15.1}
Let $\nu > \nu_c$. The process $\tilde{\A}_s$ is positive recurrent under $\PP^*$ with invariant density
\begin{align*}
\tilde{\pi}_{\nu}(x) = \tilde{c}_*x^{2\nu+1-4a}(1-x)^{2a-1}, \quad \tilde{c}_*= \frac{\Gamma(2-2 a+2 \nu )}{ \Gamma(2 a) \Gamma(2-4 a+2 \nu ) }.
\end{align*}
Consequently, if $\tilde{X}_t$ follows the same SDE as $\tilde{A}_t$ but is started according to $\tilde{\pi}_\nu$, then for $t \ge 0$ and $y > 0$,
\[
\PP^*\left\{\tilde{X}_t \le y \right\} \lesssim y^{2\nu+2-4a}, \quad \E^*[\tilde{X}_t^{-\nu}] < \infty.
\]
\end{lemma}

Note that this invariant density also appear in \cite{Law:boundary_mink}, although in a disguised form. We remark that there is a similar invariant distribution for the process in the time parameterization under which the derivative decays deterministically.

\section{First-moment estimates}
In order to compute the Hausdorff dimension of $V_{\beta}$ we need good control on the moments of $h_{\tau_s}'$ and we need to identify the event on which a given moment concentrates.

For sufficiently well-behaved Loewner curves one often has $\Delta_{\tau_s} \asymp e^{-s}$. If this was true for SLE$_\kappa$ curves then some simple algebra plus the fact that $M_{\tau_s}^{\nu}$ comes from a martingale would give
\begin{align}\label{eqn:moment_intuition}
\E \left[ h_{\tau_s}'(x)^{\lambda} \1{\tau_s < \infty} \right] & = \E \left[ \Delta_{\tau_s}(x)^{\nu} M_{\tau_s}^{\nu}(x) \1{\tau_s < \infty} \right] \nonumber  = \E_{x, \nu}^* \left[ \Delta_{\tau_s}(x)^{\nu} \right] \asymp e^{-\nu s}.
\end{align}
From this, by multifractal formalism, we expect that the expectations concentrate on the event that $h'_{\tau_s}(x) \approx e^{-\beta s}$, where $\beta(\lambda) = \nu'(\lambda)$. The latter is made precise below, but \emph{a priori} bounds of the form $\Delta_{\tau_s} \asymp e^{-s}$ are impossible for SLE curves.

We start by giving the up-to-constants estimate on the moments of $\tilde{h}_{s}'(x)$.
These bounds (with estimates on the error terms) have appeared in  \cite{Law:boundary_mink} and are very similar (but easier) than the analogous estimates for the reverse SLE flow close to the tip \cite{JL:tip}. Because of this we will be rather brief.
\begin{proposition}\label{thm:moment_ub_1}
Let $\lambda > \lambda_c = -\nu_c^2/(2a)$. For all $x \geq 1$ we have,
\begin{align*}
\E \left[ \tilde{h}_{s}'(x)^{\lambda} \1{\tilde{t}(s) < \infty} \right] \asymp x^{-\nu} e^{-\nu a s},\end{align*}
where the implicit constant depends only on $a$ and $\lambda$.
\end{proposition}
\begin{proof}
The result for $\lambda = 0$ is well-known; see \cite{AK:intersection_prob} for a proof. For other $\lambda$ let $\E^*$ refer to expectation with respect to $\PP^*_{x, \nu}$, the measure weighted by $\tilde{M}_s^{\nu}(x) = \tilde{h}'_s(x)^\lambda \tilde{\Delta}_s(x)^{-\nu} = \tilde{h}'_s(x)^\lambda e^{a\nu s} \tilde{A}_s(x)^\nu$. We have
\begin{align*}
\E[\tilde{h}'_s(x)^\lambda \1{\tilde{t}(s) < \infty} ] & = x^{-\nu} \E^*[\tilde{\Delta}_s(x)^{\nu}] = x^{-\nu} e^{-a \nu s} \E^*[\tilde{A}_s^{-\nu} \mid \tilde{A}_0 = 1].
\end{align*}
Note that since $\nu > \nu_c$ Lemma~\ref{lem:dec15.1} implies that $\E^*[X^{-\nu}] <\infty$, where $X$ is a random variable with the distribution of the invariant density of $\tilde{A}_s$. By a coupling argument (namely that $\tilde{A}$ started from the invariant distribution is always to the left of $\tilde{A}$ started from $1$) we see that $\E^*[\tilde{A}_s^{-\nu} \mid \tilde{A}_0 = 1] \le \E^*[X^{-\nu}]$ and this gives the upper bound. Since $\tilde{A} \in [0,1]$ and $\nu > 0$, $\E^*[\tilde{A}_s^{-\nu} \mid \tilde{A}_0 = 1] \ge 1$ gives the lower bound.
\end{proof}

Using again that $\crad_t(x) \asymp \dist(\gamma[0,t],x)$, with universal constants, there is a constant $c_1$ such that $\tilde{t}(s/a - c_1) \le \tau_s \le \tilde{t}(s/a+c_1)$. This, with Proposition~\ref{thm:moment_ub_1} and $t \mapsto h_t'(x)$ decreasing, implies the following.
\begin{corollary}\label{cor1}
Suppose $\lambda > 0$. For all $x \geq 1$,
\[
\E \left[ h_{\tau_s}'(x)^{\lambda} \1{\tau_s < \infty} \right] \asymp x^{-\nu} e^{-\nu s}.
\]
\end{corollary}
\begin{remark}
The work in next section implies a version of Corollary~\ref{cor1} that also holds for $\lambda < 0$.
\end{remark}
\subsection{Concentration of measure}\label{sect:concentration}
We will now study the typical behavior of $\tilde{h}'_s(1)$ under $\PP^*$, the measure weighted by the local martingale with parameters $\nu > \nu_c$ and $x=1$. The ideas are similar to those in Section~7 of \cite{Law:mf}, see also the appendix of \cite{JL:tip}.

Recall the radial time-change from the previous section. We will work with the representation
\[
 \tilde{h}'_{s}(1) = \exp \left \{ -a \int_0^{s} \tilde{A}_u^{-1}(1-\tilde{A}_u) du \right \},
\]
where $\tilde{A}_s$ follows the SDE \eqref{eq:tilde-a}; we emphasize that we consider this under $\PP^*$ which depends on $\nu$.  We have seen that the diffusion $\tilde{A}_s$ has an invariant distribution $\tilde{\pi}(x)$ under $\PP^*$. Hence if we define \[\tilde{L}_s = -\frac{1}{a}\log \tilde{h}_s'(1)=\int_0^s\tilde{A}_u^{-1}(1-\tilde{A}_u) du\] then by an ergodic theorem the time-average converges $\PP^*$-almost surely to the space average: \[
\lim_{s \to \infty} s^{-1}\tilde{L}_s = \int_0^1 x^{-1}(1-x) \tilde{\pi}(x) dx = \beta, \quad \quad (\PP^*-\text{a.s}),\]
the last identity following from a direct computation. This is yet another way in which the proof produces the value $\beta$.
We need some additional information about the smaller order terms. We will show that, roughly speaking, as $s \to \infty$, with large $\PP^*$-probability,
\[
\tilde{L}_s = \beta s  +O(s^{1/2}).
\]
Using the definition of $\tilde{L}_s$ we can rewrite the $\PP$ local martingale $\tilde{M}_s^{\nu}(1)$ as
\[
\tilde{M}_s = e^{-a \lambda \tilde{L}_s}e^{a \nu(\lambda)s} \tilde{A}_s^{\nu(\lambda)}.
\]
Here $\nu(\lambda) = \nu_c + \sqrt{\nu_c^2+2a \lambda}$. Now let $|\delta|$ be small and define a process $\tilde{N}_s$ by
\[
\tilde{N}_s = e^{-a \delta \tilde{L}_s}e^{as[\nu(\lambda+\delta) - \nu(\lambda)]} \tilde{A}_s^{\nu(\lambda + \delta) - \nu(\lambda)} = \tilde{M}_s^{\nu(\lambda + \delta)}/\tilde{M}_s^{\nu(\lambda)}.
\]
Using the last identity a straightforward computation shows that $\tilde{N}_s$ is a local martingale under $\PP^*$. The idea of the next lemma is that when $\delta$ is small, since $\beta = \nu'(\lambda)$, we have $\nu(\lambda + \delta) - \nu(\lambda) = \delta \beta+ O(\delta^2)$. So we can write the local martingale approximately as $e^{-a\delta(\tilde{L}_s - \beta s)} \tilde{A}_s^{\delta \beta}$. We will take the expectation with $\delta = \pm \ee/\sqrt{s}$ and use the super martingale property. Information about the invariant distribution of $\tilde{A}$ then gives the estimate we want.
\begin{lemma} \label{dec12.1}\label{dec12.2}
Let $\nu > \nu_c$. There is a constant $c < \infty$ such that the following holds.
For $p > 0$ sufficiently small and $t \ge 1$,
\begin{equation}\label{dec7.1}
\E^*\left[\exp\left\{ p \frac{|\tilde{L}_t-\beta t|}{\sqrt{t}} \right\} \right] \le c.
\end{equation}
Moreover, for $p > 0$ sufficiently small and $1 \le s\le t$,
\begin{equation}\label{dec7.222}
\E^*\left[\exp\left\{ p \frac{|\tilde{L}_t-\tilde{L}_s-\beta(t-s)|}{\sqrt{t-s}} \right\} \right] \le c.
\end{equation}

\end{lemma}

\begin{proof}
We rewrite $\log \tilde{N}_t - \left(\nu(\lambda + \delta) - \nu(\lambda) \right) \log \tilde{A}_t$ as
\begin{align*}
-a \delta \left( \tilde{L}_t - \beta t \right) - a\delta \beta t + at\left(\nu(\lambda+\delta) - \nu(\delta)\right) & = -a \delta \left( \tilde{L}_t - \beta t \right) + at\left(\nu'(\lambda)\delta -\delta \beta + O(\delta^2)\right).
\end{align*}
Since $\nu'(\lambda) = \beta$ we have cancellation in the second term and so
\[
\log \tilde{N}_t  = -a \delta \left( \tilde{L}_t - \beta t \right) + O(t\delta^2) + \left(\beta \delta + O(\delta^2)\right) \log \tilde{A}_t.
\]
Choose $\delta = \pm \frac{\ee}{\sqrt{t}}$, where $\ee>0$ is sufficiently small so that $\tilde{N}_t$ is well-defined. We get
\[
\log \tilde{N}_t  = -a \frac{\pm \ee}{\sqrt{t}} \left( \tilde{L}_t - \beta t \right) + O(\ee) + \left(\beta \frac{\pm \ee}{\sqrt{t}} + O(\ee^2/t)\right) \log \tilde{A}_t.
\]
Exponentiating, the (super-) martingale property implies
\begin{equation}\label{dec5.0}
\E^*\left[\exp\left\{ \mp a \ee \frac{\tilde{L}_t-\beta t}{\sqrt{t}} \right\} \tilde{A}_t^{\pm \beta \ee \frac{1}{\sqrt{t}} + O(\ee^2/t)} \right] \le c.
\end{equation}
Since $\tilde{A}_t \in [0,1]$ and $\beta, \ee >0$, from this we read off directly that for $t \ge 1$,
\begin{equation}\label{dec5.00}
\E^*\left[\exp\left\{ a \ee \frac{\tilde{L}_t-\beta t}{\sqrt{t}} \right\} \right] \le c,
\end{equation}
for $\ee > 0 $ sufficiently small. To prove \eqref{dec7.1} it remains to check that the expectation with $\ee$ replaced by $-\ee$ is bounded by a constant. For this we start from \eqref{dec5.0} which in this case reads
\[
\E^*\left[\exp\left\{-a \ee \frac{\tilde{L}_t-\beta t}{\sqrt{t}} \right\} \tilde{A}_t^{\frac{\beta \ee }{\sqrt{t}} + O(\ee^2/t)} \right] \le c.
\]
Let $y > 0$. Then the last estimate implies
\begin{equation}\label{dec5.1}
\E^*\left[\exp\left\{-a \ee \frac{\tilde{L}_t-\beta t}{\sqrt{t}} \right\} \1{\tilde{A}_t > y}\right] \le  c y^{-2\frac{\beta \ee}{\sqrt{t}}}.
\end{equation}
Moreover,
\begin{equation}\label{dec5.2}
\E^*\left[\exp\left\{-a \ee \frac{\tilde{L}_t-\beta t}{\sqrt{t}} \right\} \1{\tilde{A}_t \le y}\right] \le e^{a\ee \beta \sqrt{t}} \PP^*(\tilde{A}_t \le y) \le e^{a\ee \beta \sqrt{t}}y^q, \quad q:=2\nu+2-4a>0.
\end{equation}
Here we first used that $\tilde{L}_t \ge 0$, and since $\tilde{A}_0=1$ the probability is bounded above by the probability of the same event but where the process is started according to the invariant distribution, see Lemma~\ref{lem:dec15.1}. We see that both \eqref{dec5.1} and \eqref{dec5.2} are bounded by a constant if we choose
\[
y = e^{-\frac{a\ee \beta}{q} \sqrt{t}}.
\]
Together with \eqref{dec5.00} this completes the proof of \eqref{dec7.1}.

We now prove \eqref{dec7.222}. The argument is similar so we will sketch the proof.
Write \[\tilde{\Gamma}_t = \tilde{L}_t - \beta t.\]
We again use the martingale $\tilde{N}_t$. Rearranging the supermartingale inequality
\[
\E^*[\tilde{N}_t \mid \tilde{\mathcal{F}}_s] \le  \tilde{N}_s
\]
and then using $\nu(\lambda + \delta)-\nu(\lambda)=\beta \delta + O(\delta^2)$ as before, we get
\[
\E^* \left[e^{-a\delta (\Gamma_t - \Gamma_s)}e^{O((t-s)\delta^2)} \tilde{A}_t^{\nu(\lambda+\delta)-\nu(\lambda)} \right] \le \E^*\left[\tilde{A}_s^{\nu(\lambda+\delta)-\nu(\lambda)} \right] \le c.
\]
The last estimate uses the invariant distribution and that $\delta$ is taken sufficiently small. We apply this with $\delta = \pm \ee /\sqrt{t-s}$ and assume $t-s \ge c_0$ for some constant $c_0$ (we can do this without loss of generality since if $t-s \le c_0$ the result holds by monotonicity). This gives
\[
\E^* \left[\exp\left\{ \mp a\ee \frac{\Gamma_t - \Gamma_s}{\sqrt{t-s}} \right\} \tilde{A}_t^{\pm \beta \ee \frac{1}{\sqrt{t-s}} + O(\ee^2/(t-s))} \right] \le c.
\]
Again the more difficult case is to bound
\[\E^* \left[\exp\left\{ - a\ee \frac{\Gamma_t - \Gamma_s}{\sqrt{t-s}} \right\} \tilde{A}_t^{\beta \ee \frac{1}{\sqrt{t-s}} + O(\ee^2/(t-s))} \right],
\]
but since \[\Gamma_t-\Gamma_s = \int_s^t\tilde{A}_u^{-1}(1-\tilde{A}_u) du - \beta(t-s) \ge -\beta(t-s),\] we can estimate as before and consider the event that \[\tilde{A}_t  \le e^{-\frac{a \ee \beta}{q}\sqrt{t-s}}\] and its complement. This completes the proof.
\end{proof}

We now get the following result. It is an analogue of Proposition~7.3 of \cite{Law:mf}, as is the proof.
\begin{proposition}\label{dec15.1}
Suppose $\nu > \nu_c$. There is a constant $c < \infty$ such that the following holds. Fix $t > 0$. For $u > 0$, let $\tilde{I}_t=\tilde{I}_t^u$ be the event that the following inequalities hold for all $0\le s \le t$:
\begin{enumerate}[(a)]
\item{$|\tilde{L}_s - \beta s| \le u \sqrt{s}\log(2+s)+c,$ }
\item{$|\tilde{L}_t - \tilde{L}_s - \beta(t-s)| \le u\sqrt{t-s} \log(2+t-s) + c.$}
\end{enumerate}
Then for any $\ee > 0$ there exists $u < \infty$ such that uniformly in $t\ge 0$,
\[
\PP^*\left\{ \tilde{I}_t^u\right\} \ge 1-\ee.
\]
\end{proposition}
\begin{proof}
This follows from Lemma~\ref{dec12.1} by splitting into subintervals of length $1$. Using the fact that $\tilde{L}_s$ is increasing in $s$ it in enough to consider the values for integer $s$. Using Chebyshev's inequality we get a bound for the probability of the complement, independently of $t$, of the form $\sum_{k=1}^\infty \exp\{-u p \sqrt{k}\log(2+k) \}$, where $p>0$ is as in Lemma~\ref{dec12.1}. This quantity is $o(1)$ in $u$. The proof of the second inequality is similar.
\end{proof}
\begin{remark}
Unless otherwise stated we will from now on write $\tilde{I}_t^u$ for both the indicator of the event of Proposition~\ref{dec15.1} and the event itself, with the particular choice of $c$ guaranteed by Proposition~\ref{dec15.1} but keeping $u$ as a free parameter.
\end{remark}
\begin{remark}\label{rem:11}
Recall that $\tilde{h}_s'(1)=e^{-a\tilde{L}_s}$. Hence on the event $\tilde{I}_t^u$ of Proposition~\ref{dec15.1} we have for $0 \le s \le t$,
\[ \psi_1(s)^{-1}e^{-a\beta s} \le \tilde{h}_s'(1) \le e^{-a\beta s}\psi_1(s),\]
and
\[ \psi_1(t-s)^{-1} e^{a\beta(t-s)}\le \frac{\tilde{h}_s'(1)}{\tilde{h}_t'(1)} \le e^{a\beta(t-s)}\psi_1(t-s). \]
Here $\psi_1(x) = \exp \{ au\sqrt{x} \log(2+x) \}$.
\end{remark}
The next lemma is the only place we need to assume $\lambda > 0$ (equivalently, $\nu > \nu_0)$ in this section.
\begin{lemma}\label{dec15.3}
Suppose $\lambda > 0$ and $u > 0$. There is a sub-exponential function $\psi$ such that the following holds. Let $0 \le s \le t$. Write $\tilde{I}_t=\tilde{I}_t^u$ for the indicator of the event of Proposition~\ref{dec15.1}. Then
\[
\tilde{h}_t'(1+e^{-as})^\lambda \tilde{I}_t \le c \psi(s) e^{-a\lambda \beta s}\tilde{I}_t .
\]
\end{lemma}
\begin{proof}
By the distortion theorem there is a constant $c > 0$ such that
\[
\tilde{h}_t'(1+e^{-as})^\lambda \tilde{I}_t \le \tilde{h}_{s-c}'(1+e^{-as})^\lambda \tilde{I}_t \lesssim \tilde{h}_{s-c}'(1)^\lambda \tilde{I}_t.
\]
The first inequality uses that $\lambda > 0$. Now on the other hand, on the event $\tilde{I}_t$ we have the estimate
\[
\tilde{h}_s'(1)^\lambda \tilde{I}_t\le e^{au \lambda \sqrt{s} \log(2+s)} e^{-a\lambda \beta s}.
\]
Combining these bounds gives the result.
\end{proof}
It remains to phrase these facts in terms of $h'_{\tau_s}(1)$. Recall that there is a constant $c_1$ such that \[\tilde{t}((s/a-c_1)\wedge 0) \le \tau_s \le \tilde{t}(s/a+c_1)\] for $s > 0$. With this choice of $c_1$, and with the parameter $u > 0$, define the random variable
\begin{equation}\label{jan15.2}
I_t^u = \E\left[ \tilde{I}_{t/a+c_1}^u \mid \mathcal{F}_{\tau_t}\right],
\end{equation}
where $\tilde{I}_{t/a+c_1}^u$ is the indicator of the event of Propsosition~\ref{dec15.1}. The reason for taking a conditional expectation here is that we want to have measurability with respect to $\mathcal{F}_{\tau_t}$.
\begin{proposition}\label{prop:event}Let $\nu > \nu_c, u>0$. There exist $0<c < \infty$ and a sub-exponential function $\psi$ such that the following holds. Let $I_t=I_t^u$ be as in \eqref{jan15.2}. Then for $t \ge 0$,
\begin{equation}\label{jan15.1}
c^{-1} e^{-\nu t} \le \E\left[h'_{\tau_t}(1)^\lambda I_t \right]  \le c e^{-\nu t}.
\end{equation}
Furthermore, the following estimates hold for all $0 \le s \le t$:
\begin{enumerate}[(a)]
\item \[\psi(s)^{-1}e^{-\beta s}I_t \le h_{\tau_s}'(1)I_t \le e^{-\beta s}\psi(s)I_t;\]
\item \[h_{\tau_t}'(1) \psi(t-s)^{-1} e^{\beta(t-s)}I_t \le h_{\tau_s}'(1) I_t\le h_{\tau_t}'(1) e^{\beta(t-s)}\psi(t-s)I_t;\]
\item{If in addition we assume $\nu > \nu_0$ (so that $\lambda > 0$) then it also holds that \[h'_{\tau_t}(1+e^{-s})^\lambda I_t \le \psi(s) e^{- \lambda \beta s} I_t.\]}
\end{enumerate}
\end{proposition}

\begin{proof}
We prove \eqref{jan15.1} first. Let $u > 0$ be fixed for the moment. By Proposition~\ref{dec15.1},
\begin{align}\label{feb26.1}
\tilde{h}_{t/a-c_1}'(1) \tilde{I}_{t/a+c_1} \asymp \tilde{h}_{t/a+c_1}'(1) \tilde{I}_{t/a+c_1},
\end{align}
where the implicit constants depend on $u$. Since $\tilde{h}_{t/a+c_1}'(1) \le h'_{\tau_t}(1) \le \tilde{h}_{t/a-c_1}'(1) $ this implies that
\begin{align*}
\E \left[ h'_{\tau_t}(1)^\lambda I_t \right] = \E \left[ h'_{\tau_t}(1)^\lambda \tilde{I}_{t/a+c_1} \right]\asymp \E \left[\tilde{h}_{t/a+c_1}'(1)^\lambda \tilde{I}_{t/a+c_1} \right].
\end{align*}
Using Proposition~\ref{thm:moment_ub_1}, the upper bound in \eqref{jan15.1} follows immediately. For the lower bound we recall that $\tilde{A} \in [0,1]$ and $\nu > 0$, so
\begin{align*}
 e^{\nu t}\E\left[\tilde{h}'_{t/a+c_1}(1)^\lambda \tilde{I}_{t/a+c_1} \right] \ge c \E\left[\tilde{h}'_{t/a+c_1}(1)^\lambda e^{\nu(t+c_1)} \tilde{A}_{t/a+c_1}^\nu \tilde{I}_{t/a+c_1} \right] \ge c_2 \PP^*\left\{ \tilde{I}_{t/a+c_1}\right\},
\end{align*}
where the constants do not depend on $u$.
We now choose $u$  (somewhat arbitrarily) as the smallest number such that the last probability is at least $1/2$, independently of $t$. This completes the proof of \eqref{jan15.1}. Remark~\ref{rem:11} then shows that the estimates $(a)$ and $(b)$ hold, changing, if necessary, the sub-exponential function $\psi$.

Finally, if we assume that $\nu > \nu_0$, so that $\lambda > 0$, we can apply Lemma~\ref{dec15.3} to see that $(c)$ holds.
\end{proof}
\begin{remark}
It is clear from the proofs that the same statements, possibly allowing for a different subpower function and $u$, hold if we replace $1$ by any $y \in [1,2]$.
\end{remark}
\begin{remark}
In order to derive the harmonic measure spectrum discussed in the introduction it essentially suffices to study the decay of the process
\[
\tilde{h}_{s}(x) = \tilde{A}_s(x)^{-1} e^{-as} \tilde{h}'_{s}(x), \quad s \to \infty.
\]
We will not discuss details, but remark that the arguments in this section give the required control of the behavior of the non-trivial quantities on the right-hand side under the weighted measure $\PP^*$. \end{remark}
\section{Correlation estimate}
We will carry out the two-point estimate using the ``direct'' method. This is significantly easier in the boundary case since the SLE path can get close to two points in ``fewer'' ways than in the bulk case. The proof is similar to the one for the two-point estimate in Section 2 of \cite{Law:boundary_mink}, though the situation here is a bit more complicated. Throughout this section we assume that all parameters are chosen in a matching way, i.e., that $\beta = \beta(\lambda), \nu=\nu(\lambda)$, etc, but with the restriction that $\lambda > 0$. We write $d = 1+\lambda \beta -\nu$ which we assume is strictly between $0$ and $1$.

Given $u > 0$ and $x \in [0,1]$ we write $I_t(1+x) = I_t^u(1+x)$ as in Proposition~\ref{prop:event} with the point $1$ replaced by $1+x$. (We usually suppress the $u$ in order to keep notation lighter.) Throughout we fix $u$ and $\psi$ so that the conclusion of Proposition~\ref{prop:event} holds for these points.
\begin{theorem}\label{thm:correlation}
Let $\lambda>0$ be chosen such that $d > 0$. There exists $u>0$ and a subpower function $\psi$ such that for $n \in \mathbb{N}$ and $x \ge 0$
\[
\E \left[ h'_{\tau_n}(1)^\lambda \, h'_{\sigma_n}(1+x)^\lambda \, I_n(1) \, I_n(1+x) \right] \le  e^{-2n\nu} \left(x \vee e^{-n} \right)^{\lambda \beta -\nu}\psi \left(1/\left(x \vee e^{-n} \right) \right),
\]
where $\tau_n = \tau_n(1),\, \sigma_n=\tau_n(1+x)$.
\end{theorem}
\begin{remark}
If we did not include the $I$-variables in Theorem~\ref{thm:correlation}, then as $1+x \to 1+$ the expectation would concentrate on the event that $h_{\tau_n}'(1) \approx e^{-n \beta(2\lambda)}$, which does not give the ``correct'' bound except in the case $\lambda = 0$. In fact, in the latter case one can directly use a two-point martingale to prove the two-point estimate. There is a family of such martingales for all $\lambda$ we consider, but we have not found a way to use them to prove Theorem~\ref{thm:correlation}.
\end{remark}
For a stopping time $\tau$, we write $\gamma_\tau$ both for $\gamma[0,\tau]$ and for the filtration generated by the curve up to time $\tau$. The following lemma is the analogue of Proposition~2.1 of \cite{Law:boundary_mink}.
\begin{lemma} \label{1pt_1}
Let $\lambda > 0$. Suppose $\tau$ is a stopping-time with $\tau \le \tau_{n-1}$ almost surely. Then,
\begin{align*}
\E \left[ h_{\tau_{n}}'(1)^\lambda \mid  \gamma_\tau \right]
 & \lesssim h_\tau'(1)^{\lambda} \left( \frac{e^{-n} }{\dist(\gamma_\tau, 1)} \right)^{\nu}.
\end{align*}
\end{lemma}
\begin{proof}
Let $B \subset \overline \H$ be the semi-disc of radius $e^{-n}$ about $1$. For $s \le t$ write $h_{s,t}(w) = h_t \circ h_s^{-1}(w)$ so that
\[
h_t(z) = h_{t,s} \circ h_s(z).\]
Then by the chain rule $h'_{\tau_n}(1) = h_\tau'(1) h_{\tau,\tau_n}'(h_\tau(1))$. Since $\tau \le \tau_{n-1}$ distortion estimates imply that
\[
\text{diam}( h_\tau(B) ) \asymp e^{-n} h'_\tau(1).
\]
Using the definition \eqref{eqn:Ct_defn} of $C_t$, we have
\[
\dist(\gamma_\tau, 1) \asymp C_\tau \le \frac{h_\tau(1)}{h'_\tau(1)}.
\]
Hence, as $\lambda, \nu > 0$, by the domain Markov property, scaling, and Corollary~\ref{cor1},
\[
\E \left[ h_{\tau,\tau_n}'(h_\tau(1))^\lambda \1{\tau_n < \infty}\mid \gamma_\tau \right] \lesssim \left( \frac{e^{-n} h'_\tau(1) }{h_\tau(1)} \right)^\nu \lesssim \left( \frac{e^{-n}}{\dist(\gamma_\tau, 1)} \right)^\nu.
\]
\end{proof}
It is useful to have a version of the one-point estimate expressed in terms of a conformally invariant quantity. We will use \textit{excursion measure}, but one could equally well use the more standard extremal length. If $D$ is a domain with analytic boundary and $A,B \subset \partial D$, then excursion measure is defined by  \[\mathcal{E}_D(A,B) = \int_A \partial_n \omega_D(\zeta, B) |d\zeta| =  \int_B \partial_n \omega_D(A, \zeta) |d\zeta|,\]
where $\omega$ is harmonic measure and $\partial_n$ is the normal derivative in the inward-pointing direction. It is easy to see that this is a conformal invariant and so we may define excursion measure in rough domains by conformal invariance. If $\eta$ is a crosscut of $\mathbb{H}$ whose endpoints are both positive, then it is well-known that
\[
\mathcal{E}_{\mathbb{H} \setminus \eta}(\mathbb{R}_-,\eta) \wedge 1 \asymp \frac{\operatorname{diam} \eta}{\dist(\eta, 0)} \wedge 1.
\]
\begin{lemma} \label{1pt_2}
Let $\eta$ be a crosscut separating $x > 0$ from $\infty$, with $0$ and $x$ in different components of $\H \setminus \eta$.  Let $\tau_\eta = \inf \{t \ge 0: \gamma[0,t] \cap \eta \neq \emptyset\}$ be the hitting time of $\eta$. If $\lambda > 0$, then
\[
\E \left[ h_{\tau_\eta}'(x)^\lambda \right] \lesssim \mathcal{E}_{\mathbb{H} \setminus \eta}\left(\mathbb{R}_-,\eta \right)^\nu.
\]
\end{lemma}
\begin{proof}
Set $\delta : = \operatorname{diam} \eta$. We can without loss of generality assume $\delta < x$, since $\mathcal{E}_{\mathbb{H}}\left(\mathbb{R}_-,\eta \right)$ is bigger than a universal constant in the other case and $h_{\tau_\eta}'(x)^\lambda \le 1$. The half-circle of radius $\delta$ around $x$ separates $\eta$ from $\infty$ in $\mathbb{H}$, so \[\tau:=\tau_{-\log \delta}(x) \le \tau_\eta(x).\] Moreover, $t \mapsto h_t'(x)$ is decreasing, so using that $\lambda > 0$,
\[
h_{\tau_\eta}'(x)^\lambda \le h_{\tau}'(x)^\lambda.
\]
From the one-point estimate we then get
\[
\E \left[ h_{\tau}'(1)^\lambda \right] \lesssim \left(\frac{\delta}{x} \right)^{\nu} \lesssim \left( \frac{\operatorname{diam} \eta}{\dist(\eta, 0)} \right)^\nu \wedge 1 \lesssim  \mathcal{E}_{\mathbb{H} \setminus \eta}(\mathbb{R}_-,\eta)^\nu \wedge 1 .
\]
\end{proof}
We will apply the last lemma in the following setting. Suppose $\tau < \tau_n$ is a stopping time such that $B=B_{e^{-n}}(1) \cap \mathbb{H}$ is inside $H_\tau$. Write $\eta = h_{\tau}(\partial B)$. This is a crosscut of $\mathbb{H}$. By the domain Markov property of SLE we can estimate
\[
\E \left[h'_{\tau_n}(1)^\lambda \mid \gamma_\tau \right] = h_{\tau}'(1)^\lambda \E \left[h_{\tau, \tau_n}'(h_\tau(1))^\lambda \mid \gamma_{\tau} \right] \lesssim h_{\tau}'(1)^\lambda \mathcal{E}_{H_\tau \setminus B}\left[ g_\tau^{-1}\left(\mathbb{R}_- \right), B \right]^\nu .\]
We now turn to the proof of Theorem \ref{thm:correlation}, which will make heavy use of the last two lemmas, Lemma \ref{dec15.3}, and Proposition \ref{prop:event}. Note that in the latter two results $\psi$ was a sub-exponential function, whereas in Theorem \ref{thm:correlation} we have changed the notation to make it a subpower function.

\begin{proof}[Proof of Theorem~\ref{thm:correlation}]
Throughout we will write $I'_n:=I_n(1)I_n(1+x)$.
When $0 \le x \le e^{-n}$, $\sigma_n=\tau_n(1+x) \ge \tau_{n-2}(1)$. Since $\lambda>0$ we can use distortion estimates to see that
\begin{align*}
\E \left[ h'_{\tau_n}(1)^\lambda \, h'_{\sigma_n}(1+x)^\lambda \, I_n' \right] & \lesssim \E\left[h'_{\tau_{n-2}}(1)^{2\lambda} I_{n}(1) \right]   \lesssim \psi(e^n) e^{-\lambda \beta n}\E\left[h'_{\tau_n}(1)^{\lambda} \right] \lesssim \psi(e^n) e^{-(\nu+\lambda \beta) n}.
\end{align*}
We assume $x \ge e^{-n}$ from now on.
First choose $r$ so that $e^{-r} \le x \le e^{-r+1}$. We may without loss of generality assume $n \ge r+5$. We write $\tau_k=\inf\{t \ge 0: \dist(\gamma_t,1) \le e^{-k}\}$ as usual and define $\sigma_k$ similarly but with $1$ replaced by $1+x$.

We will split the expectation into several events in a manner similar to Section 2.1 of \cite{Law:boundary_mink}. Note that if either event $\{ \tau_n = \infty \}$ or $\{\sigma_n = \infty\}$ occurs then, since $\lambda > 0$, at least one of $h'_{\tau_n}(1)^\lambda, h'_{\sigma_n}(1+x)^\lambda$ (interpreted as limits) is equal to $0$, so \[\E\left[ h'_{\tau_n}(1)^\lambda h'_{\sigma_n}(1+x)^\lambda  \,  1_{\{ \tau_n = \infty\} \cup \{\sigma_n = \infty\}} \, \right] = 0\] and the desired upper bound holds trivially. We thus consider the complement of this event, \[N_n = \{\tau_n < \infty\} \cap \{\sigma_n < \infty\}.\] We now give the definitions and briefly explain how we estimate the expectation on each of the events.
\begin{itemize}
\item{
Define
\[A_n=  \{\tau_n < \sigma_{r+4}\} \cap N_n .\] That is, $A_n$ is the event that the path reaches $B(1,{e^{-n}})$ before $B(1+x, e^{-(r+4)})$. In this case we will use Lemma~\ref{1pt_1}.}
\item{Next, we define
\[B_n=  \{\sigma_n < \tau_n \} \cap N_n.
\]This is the event that the path gets very close to $1+x$ before getting close to $x$. In this case we can use the fact that the excursion measure of $B(1,{e^{-n}})$ in $H_{\sigma_n}=\mathbb{H}\setminus \gamma_{\sigma_n}$ is at most $e^{-2n} / x^2$. This follows easily from a harmonic measure estimate.}
\item{Finally, for $k=r+4, \ldots, n-1,$ let \[
B_{k,n}=  \{\sigma_k < \tau_n < \sigma_{k+1}\} \cap N_n .
\]
This is the event that the path gets to distance $e^{-k}$ from $1+x$, but not closer, before getting to distance $e^{-n}$ from $1$,  and after this it gets very close (i.e. distance $e^{-k-1}$) to $1+x$. We again use the invariant one-point estimate.}
\item{
We will also write $A_n, B_n$, etc.,  for the indicators of these events.}
\end{itemize}
\emph{The estimate on $A_n$:} Recall that when $x > e^{-n}$,  there is a subpower function $\psi$ such that \[ h_{\tau_n}'(1+x)^\lambda I_n(1) \le \psi(1/x) x^{\lambda \beta}.\] See Proposition~\ref{sect:concentration}. Note also that Lemma~\ref{1pt_1} with the stopping time $\tau_n \wedge \sigma_{r+4}$ gives
\[
A_n \E \left[h'_{\sigma_n}(1+x)^\lambda \mid \gamma_{\tau_n \wedge \sigma_{r+4}} \right]  \lesssim \left(\frac{e^{-n}}{e^{-r}}\right)^\nu h'_{\tau_n}(1+x)^\lambda.
\]
Hence, we can estimate
\begin{align*}
\E \left[h'_{\tau_n}(1)^\lambda  h'_{\sigma_n}(1+x)^\lambda I_{n}' A_n \right]
 & = \E \left[ \E \left[ h'_{\tau_n}(1)^\lambda h'_{\sigma_n}(1+x)^\lambda I_{n}' A_n \mid \gamma_{\tau_n \wedge \sigma_{r+4}} \right] \right] \\
 & \le  \E \left[h'_{\tau_n}(1)^\lambda I_n(1) A_n \E \left[h'_{\sigma_n}(1+x)^\lambda \mid \gamma_{\tau_n \wedge \sigma_{r+4}} \right] \right] \\
 & \lesssim \left(\frac{e^{-n}}{e^{-r}}\right)^\nu\E \left[h'_{\tau_n}(1)^\lambda I_n(1) h'_{\tau_n}(1+x)^\lambda \right] \\
 & \lesssim \left(\frac{e^{-n}}{e^{-r}}\right)^\nu \E \left[ h'_{\tau_n}(1)^\lambda \right] \psi(e^r) e^{-r  \lambda \beta } \\
 & \lesssim e^{-\nu (n + n)} e^{-r(\lambda \beta -\nu)}\psi(e^r).
\end{align*}

\emph{The estimate on $B_n$.} Here we will use that $x > 0$ implies $h_t'(1) \le h_t'(1+x)$ and so if $\lambda > 0$ we have the bound
\[I_n(1+x)h_{\sigma_n}'(1)^\lambda \le e^{-\lambda \beta n} \psi(e^n).\]

We will also use the invariant one-point estimate, Lemma~\ref{1pt_2}, which implies
\[
B_n\E\left[ h'_{\tau_n}(1)^\lambda    \mid \gamma_{\tau_n \wedge \sigma_{n}} \right] \lesssim B_n  \,h'_{\sigma_n}(1)^\lambda  \mathcal{E}_{H_{\sigma_n}}\left[g_{\sigma_n}^{-1}(\mathbb{R}_-), B(1, e^{-n}) \right]^\nu \lesssim h'_{\sigma_n}(1)^\lambda  \left( \frac{e^{-2n}}{e^{-2r}} \right)^\nu.
\]
Using these bounds we can estimate as follows:
 \begin{align*}
\E[h'_{\tau_n}(1)^\lambda  h'_{\sigma_n}(1+x)^\lambda I_n' B_n] &=  \E \left[ \E \left[ h'_{\tau_n}(1)^\lambda h'_{\sigma_n}(1+x)^\lambda I_n' B_n \mid \gamma_{\tau_n \wedge \sigma_{n}} \right] \right] \\
& \le \E \left[   h'_{\sigma_n}(1+x)^\lambda I_n(1+x) B_n  \E  \left[h'_{\tau_n}(1)^\lambda    \mid \gamma_{\tau_n \wedge \sigma_{n}} \right] \right] \\
& \lesssim \left(\frac{e^{-n}e^{-n}}{e^{-2r}}\right)^\nu \E \left[ I_n(1+x)  h'_{\sigma_n}(1+x)^\lambda   h'_{\sigma_n}(1)^\lambda\right] \\
& \le \psi(e^n) e^{-\lambda \beta n}\left(\frac{e^{-n}e^{-n}}{e^{-2r}}\right)^\nu \E \left[h'_{\sigma_n}(1+x)^\lambda   \right] \\
& \le \psi(e^n) e^{-\lambda \beta n}\left(\frac{e^{-n}e^{-n}}{e^{-2r}}\right)^\nu e^{-\nu n}\\
& =\psi(e^n) e^{- 2\nu n} e^{-\lambda \beta n} e^{r \nu} e^{-(n-r) \nu}\\
& \le  e^{-2 \nu n} e^{-r(\lambda \beta -\nu)} \left(\frac{e^{-n}}{e^{-r}}\right)^{\lambda \beta + \nu} \psi(e^{-n}) \\
& \le e^{-2 \nu n}e^{-r(\lambda \beta - \nu)} \psi_1(e^r).
\end{align*}
The last step uses that if $0<\delta \le c x$, then there is a subpower function $\psi_1$ such that
\[\sup_{0 < \delta \le c x}(\delta/x)^{\lambda \beta + \nu} \psi(1/\delta) = \sup_{0 < \delta \le c x}(\delta/x)^{\lambda \beta + \nu} \psi( (x/\delta) \cdot 1/x)  \le \psi_1(1/x).\]
\emph{The estimate on $B_{k,n}$.} We will need to sum over $r + 4 \le k \le n-1$.
We first note that from the domain Markov property, the invariant one-point estimate, and monotonicity of $x \mapsto h_t'(x)$,
\begin{align}\label{oct15.1}
B_{k,n} \E\left[h'_{\tau_n}(1)^\lambda  \mid \gamma_{\sigma_k} \right] & \lesssim  h'_{\sigma_k}(1)^\lambda \left( \frac{e^{-n} e^{-k}}{e^{-2r}} \right)^\nu  \lesssim h'_{\sigma_k}(1+x)^\lambda \left( \frac{e^{-n} e^{-k}}{e^{-2r}} \right)^\nu.
\end{align}
Moreover, using Lemma~\ref{1pt_1},
\begin{align}\label{oct15.2}
\E\left[h_{\sigma_n}'(1+x)^\lambda I_n' B_{k,n} \mid \gamma_{\tau_n} \right] & \lesssim I_n(1) h'_{\tau_n}(1+x)^\lambda \left( \frac{e^{-n}}{e^{-k}} \right)^\nu B_{k,n} \le \psi(e^r) e^{-\lambda \beta r} \left( \frac{e^{-n }}{e^{-k}} \right)^\nu B_{k,n}.
\end{align}
Putting these two estimates together we get
\begin{align*}
 \E \left[ h_{\tau_n}'(1)^\lambda h_{\sigma_n}'(1+x)^\lambda  I_n' B_{k,n} \right] & = \E \left[  h_{\tau_n}'(1)^\lambda \E\left[  h_{\sigma_n}'(1+x)^\lambda I_n' B_{k,n} \mid \gamma_{\tau_n} \right] \right] \\
 & \lesssim \E \left[  h_{\tau_n}'(1)^\lambda B_{k,n} \right]  e^{-\lambda \beta r} \left( \frac{e^{-n}}{e^{-k}} \right)^\nu \psi(e^r) \\
 & =  \E \left[ \E \left[  h_{\tau_n}'(1)^\lambda B_{k,n} \mid \gamma_{\sigma_k} \right] \right]  e^{-\lambda \beta r} \left( \frac{e^{-n}}{e^{-k}} \right)^\nu \psi(e^r)\\
 & \lesssim \E \left[ h'_{\sigma_k}(1+x)^\lambda \right]  \left( \frac{e^{-n} e^{-k}}{e^{-2r}} \right)^\nu  e^{- \lambda \beta r} \left( \frac{e^{-n}}{e^{-k}} \right)^\nu \psi(e^r) \\
 & \lesssim e^{-\nu k }\left( \frac{e^{-n} e^{-k}}{e^{-2r}} \right)^\nu  e^{- \lambda \beta r} \left( \frac{e^{-n}}{e^{-k}} \right)^\nu \psi(e^r).
\end{align*}
If we use a union bound and sum this over $k$ from $r+4$ to $n-1$ we get our desired estimate:
\[
\E [ h_{\tau_n}'(1)^\lambda h_{\sigma_n}'(1+x)^\lambda  I_n' \cup_{k=r+4}^{n-1} B_{k,n} ] \lesssim e^{-2 \nu n} e^{-(\lambda \beta - \nu) r } \psi(e^{r}).
\]
The proof is complete.
\end{proof}

\section{Dimension spectrum}

In this section we focus on using the moment estimates to prove Theorem \ref{thm:main}. The proof spans the whole section. First we show that the dimension of the sets must almost surely be constant.

\begin{lemma}
For each $\beta$ the quantity $\hdim V_{\beta}$ is constant almost surely.
\end{lemma}

\begin{proof}
For $x > 0$ let $S_x = V_{\beta} \cap (0,x)$. Using scaling properties of SLE it can be shown that for any $x > 0$ the random set $S_x$ has the same law as $x S_1$. Since Hausdorff dimension is unchanged by linear scaling we have $\hdim xS_1 = \hdim S_1$, and therefore $\hdim S_x$ has a law which doesn't depend on $x$. But $\hdim S_x$ has an almost sure limit as $x \to 0+$ (since the sets are decreasing) and this limit is $\F_{0+}$-measurable since each $S_x$ is $\F_{T_x}$-measurable. Since the filtration is that of a Brownian motion, the Blumenthal $0$-$1$ law implies that the limit is almost surely constant. But each $\hdim S_x$ has the same law, and therefore must be that same constant for all $x > 0$.
\end{proof}

To compute the constant we will use a slightly different version of the sets $V_{\beta}$. Observe that by monotonicity of the derivative $t \mapsto h_t'(x)$ we may modify the definition of $V_{\beta}$ by replacing the real variable $s \to \infty$ with an integer variable $n \to \infty$. Therefore it is sufficient to compute the almost sure Hausdorff dimension of the random sets
\begin{align*}
V_{\beta} = \left \{ x \in \R_+ : \lim_{n \to \infty} \frac{1}{n} \log h'_{\tau_n}(x) = -\beta, \tau_n = \tau_n(x) < \infty \,\, \forall \,\, n \right \}.
\end{align*}

\subsection{Upper bound on dimension \label{sec:upper_bound}}

For the upper bound on the dimension of $V_{\beta}$ we define sets $\overline{V}_{\beta}$ and $\underline{V}_{\beta}$ by
\begin{align*}
\overline{V}_{\beta} &= \left \{ x \in \R_+ : h_{\tau_n}'(x) \geq e^{-\beta n}, \tau_n = \tau_n(x) < \infty \textrm{ i.o. in } n \right \}, \\
\underline{V}_{\beta} &= \left \{ x \in \R_+ : h_{\tau_n}'(x) \leq e^{-\beta n}, \tau_n = \tau_n(x) < \infty \textrm{ i.o. in } n \right \}.
\end{align*}
Clearly $V_{\beta}$ is a subset of both sets. Recall that $d(\beta) = 1 + \lambda \beta - \nu$. We will first prove the following:

\begin{theorem}\label{thm:upper_bound}
Suppose $\kappa > 4$. Then
\begin{enumerate}[(i)]
\item if $\beta \in [\beta_{-}, \hat{\beta}]$ then $\hdim \overline{V}_{\beta} \leq d(\beta)$ a.s.,
\item if $\beta < \beta_{-}$ then $\overline{V}_{\beta} = \emptyset$ a.s.,
\item if $\beta \in [\hat{\beta}, \beta_+]$ then $\hdim \underline{V}_{\beta} \leq d(\beta)$ a.s.,
\item if $\beta > \beta_+$ then $\underline{V}_{\beta} = \emptyset$ a.s.
\end{enumerate}
\end{theorem}

\begin{remark}\label{rem:ub_simplification}
It is sufficient to prove the theorem for $\overline{V}_{\beta} \cap J$ for every closed subinterval $J$ of $\R_+$. For the sake of concreteness we will carry out the proof for $J = [1,2]$, but it will be clear that the same arguments work for an arbitrary interval $J$ with only slightly modified constants that do not affect the conclusions of the argument.
\end{remark}

To prove these upper bounds we construct suitable covers of the sets. We first describe the cover for $\overline{V}_{\beta} \cap [1,2]$. For each $n \geq 1$ let
\begin{align*}
J_{j,n} = [1 + j e^{-n}/2, 1 + (j+1) e^{-n}/2), \quad j = 1, 2, \ldots 2 e^{n} - 1
\end{align*}
so that $J_{j,n}$ is an interval of length $|J_{j,n}| = e^{-n}/2$. For each $n$ we write $\mathcal{J}_n$ for the set of $\{J_{j,n}\}_j$. Let $x_{j,n}$ be the midpoint of $J_{j,n}$. By standard distortion estimates the derivative $h_{\tau_n(x)}'(x)$ does not vary much in a small neighborhood of $x$, that is, there is a universal constant $c_0 > 0$ such that
\begin{align}\label{eqn:deriv_neighbor}
h_{\tau_n(x)}'(x) \geq e^{-\beta n} \textrm{ for some } x \in I_{j,n} \implies h_{\tau_n(x)}'(x_{j,n}) \geq c_0 e^{-\beta n}.
\end{align}
Moreover, for simple geometric reasons the curve must hit the ball of radius $e^{-(n-2)}$ centered at $x_{j,n}$ before it hits any ball of radius $e^{-n}$ centered at a point in $J_{j,n}$, so therefore
\begin{align*}
\tau_{n-2}(x_{j,n}) \leq \min_{x \in I_{j,n}} \tau_n(x).
\end{align*}
Combining this with the fact that $t \mapsto h_t'(x_{j,n})$ is a decreasing function implies that \eqref{eqn:deriv_neighbor} reduces to
\begin{align*}
h_{\tau_n(x)}'(x) \geq e^{-\beta n} \textrm{ for some } x \in J_{j,n} \implies h_{\tau_{n-2}(x_{j,n})}'(x_{j,n}) \geq c_1 e^{-\beta (n-2)},
\end{align*}
with $c_1 = c_0 e^{-2 \beta}$. Defining now $J_{n,-}(\beta)$ by
\begin{align*}
J_{n, -}(\beta) = \bigcup J_{j,n},
\end{align*}
where the union is over all $j \in \{1, 2, \ldots, 2e^{n} - 1 \}$ such that $h_{\tau_{n-2}(x_{j,n})}'(x_{j,n}) \geq c_1 e^{-\beta(n-2)}$, then the latter implication becomes
\begin{align*}
\{ x \in [1,2] : h_{\tau_n(x)}'(x) \geq e^{-\beta n} \} \subset J_{n, -}(\beta).
\end{align*}
This inclusion and the definition of $\overline{V}_{\beta}$ implies that
\begin{align}\label{eqn:V_over_cover}
\overline{V}_{\beta} \cap [1,2] \subset \bigcup_{n \geq m} J_{n, -}(\beta),
\end{align}
for every  $m$. We now have the following:

\begin{lemma}\label{lemma:n_over_bound}
Let $\overline{\mN}_n=\overline{\mN}_n(\beta)$ be the number of $J_{j,n}$ intervals whose union forms $J_{n, -}(\beta)$, and assume $\kappa > 4$ and $\beta < \hat{\beta}$. Then
\begin{align*}
\E \left[ \overline{\mN}_{n} \right] \lesssim e^{n(1 + \lambda \beta - \nu)}.
\end{align*}
\end{lemma}

\begin{proof}
By Proposition \ref{thm:moment_ub_1} and the Chebyshev inequality we have
\begin{align*}
\E \left[ \overline{\mathcal{N}}_n\right] =  \sum_{j=3}^{2e^n-3}\PP\left\{h'_{\tau_{n}}(x_{j,n}) \ge c_0'e^{-\beta n}; \, \tau_{n} < \infty \right\} \lesssim  e^{\lambda \beta n} \sum_{j=3}^{2e^n-3} \E\left[h'_{\tau_{n}}(x_{j,n})^\lambda \, \1{\tau_{n} < \infty} \right] \lesssim e^{n(1+\lambda \beta - \nu)}.
\end{align*}
\end{proof}

From this we easily derive the following:

\begin{proof}[Proof of Theorem~\ref{thm:upper_bound}, (i) and (ii)]
First note that Lemma~\ref{lemma:n_over_bound} implies that for each $s > d(\beta)$ and $n$ sufficiently large
\begin{align*}
\E\left[ \overline{\mathcal{N}}_n(\beta)\right] \lesssim e^{ns}.
\end{align*}
But the $s'$-dimensional Hausdorff measure of $\overline{V}_{\beta} \cap [1,2]$ is bounded above by the $s'$-dimensional content of the cover provided in \eqref{eqn:V_over_cover}. Taking expectations this gives
\begin{align*}
\E\left[\mathcal{H}^{s'}\left(\overline{V}_{\beta} \cap [1,2] \right)\right]  \le c \lim_{m \to \infty} \sum_{n=m}^\infty \E\left[\overline{\mathcal{N}}_n(\beta) \right]e^{-ns'}=0,
\end{align*}
for $s' > s$. This implies that $\mathcal{H}^{s'}\left(\overline{V}_{\beta} \cap [1,2] \right)=0$ almost surely, and by Remark \ref{rem:ub_simplification} part (i) of the Theorem follows. Part (ii) holds since for $\beta < \beta_{-}$ we have $1+\lambda \beta - \nu <0$, and so
\[
\PP\left\{ \overline{V}_{\beta} \cap[1,2] \neq \emptyset \right\} \le \lim_{m \to \infty} \sum_{n = m}^{\infty} \E \left[ \overline{\mathcal{N}}_n(\beta) \right] = 0.
\]
\end{proof}

We now describe the cover for $\underline{V}_{\beta} \cap [1,2]$. This case corresponds to $\lambda < 0$ and although we could use the comment after Proposition~\ref{prop:event} we will use the moment bound on $\tilde{h}_{s}'$ directly. Suppose that $x \in \underline{V}_{\beta}$. Let $n \geq 1$ be such that $h_{\tau_n(x)}'(x) \leq e^{-\beta n}$. There is a universal constant $c_1$ such that with $s_n : = n/a+c_1$ this implies $\tilde{h}_{s_n}'(x) \le e^{-\beta n}$. Distortion estimates show that there is a smallest integer $c_2$ such that if $J$ is the unique interval in $\mathcal{J}_{n + c_2}$ with $x \in J$ then for all $z \in J$ we have $h_{\tilde{t}(s_n)}'(z) \asymp \tilde{h}_{s_n}'(x)$. In particular $h_{\tilde{t}(s_n)}'(x_J) \lesssim e^{-\beta n}$; here $\tilde{t}(s_n) = \tilde{t}_x(s_n)$. Moreover, there is a universal $c_3$ such that $\tilde{t}_{x_J}(n/a+c_3) \ge \tilde{t}_x(s_n)$ and this implies that $h_{\tilde{t}_{x_J}(n/a+c_3)}'(x_J) \le c_4 e^{-\beta n}$. We thus conclude that there are universal constants $c_2,c_3,c_4$ such that any $x \in \underline{V}_{\beta}$ is contained in an interval of $\mathcal{J}_{n+c_2}$ and $\tilde{h}_{n/a+c_3}'(x_J) \le c_4 e^{-\beta n}$ holds where $x_J$ is the midpoint of $J$.
åThus we can choose universal constants $c_5, c_6$ such that the following sets cover $\underline{V}_{\beta} \cap [1,2]$. Define
\begin{align*}
J_{n, +}(\beta) = \bigcup J_{j,n}
\end{align*}
where the union is over all $j \in \{ 1, 2, \ldots, 2 e^{n} - 1 \}$ such that $\tilde{h}_{n/a+c_5}'(x_{j,n}) \leq c_6 e^{-\beta n}$. Therefore
\begin{align*}
\underline{V}_{\beta} \subset \bigcup_{n \geq m} J_{n,+}(\beta)
\end{align*}
for every $m$. From this we have the following:
\begin{lemma}\label{lemma:n_under_bound}
Let $\underline{\mathcal{N}}_{\beta}$ be the number of $J_{j,n}$ intervals whose union forms $J_{n, +}(\beta)$ and assume that $\kappa > 4$ and $\beta > \hat{\beta}$. Then there exists a constant $c = c(\kappa) < \infty$ such that
\begin{align*}
\E[\underline{\mathcal{N}}_{\beta}] \lesssim e^{n (1 + \lambda \beta - \nu)}.
\end{align*}
\end{lemma}

The proof is the same as for Lemma \ref{lemma:n_over_bound}, and then the proof for parts (iii) and (iv) of Theorem \ref{thm:upper_bound} is the same as for parts (i) and (ii).

\subsection{Lower bound on dimension \label{sec:lower_bound}}
We now assume $\lambda > 0$ with matching values of $\nu, \beta$ such that $d \in (0,1)$.
The proof of the lower bound on dimension is now rather standard and uses Frostman's lemma, see, e.g., \cite{lawler:bolyai} for a similar construction. We construct the Frostman measure as follows.
Let $J_{j,n} = [1+(j-1)e^{-n}, 1+je^{-n})$ and let $x_{j,n}$ be the midpoint of $J_{j,n}$. Fix $u < \infty$ and a subpower function $\psi$ as in Proposition~\ref{prop:event}. Let $I_n(x)=I_n(x, \psi, u)$ be as in Proposition~\ref{prop:event} with $t=n$ and $1$ replaced by $x$. Let $\mu_{j,n}$ be the random measure on $[1,2]$ defined by
\[
\mu_{j,n}(dx) = e^{\nu n} h_{\tau_n}'(x_{j,n})^{\lambda} I_n(x_{j,n}) \chi_{J_{j,n}} dx, \quad n=1,2,\ldots
\]
From Proposition~\ref{prop:event} we have $\E \left[\mu_{j,n}([1,2]) \right] = \E \left[\mu_{j,n}(J_{j,n}) \right] \asymp |J_{j,n}| = e^{-n}$. We set
\[
\mu_n(dx) = \sum_{j=1}^{e^n-1}\mu_{j,n}(dx),
\]
and note that $\E \left[\mu_n\left( [1,2] \right) \right] \asymp 1$. We claim that there is an event of positive probability on which we may take a subsequential limit of the $\mu_n$ to obtain a Frostman measure on $V_{\beta}$. From the definition of $I_n$ it is clear that any such limiting measure has a support contained in the set $V_\beta$. We first show the existence of the event of positive probability on which we may take a non-trivial subsequential limit. By the estimates above we have that $\E\left[\mu_n\left( [1,2] \right) \right] \ge c_0 > 0$ and we will soon prove that $ \E \left[ \mu_n\left( [1,2] \right) ^2 \right] \le c_1 < \infty$. Assuming the latter bound holds we have by Cauchy-Schwarz that
\[\E\left[ \mu_n\left( [1,2] \right) \right]^2 \le \E \left[ \mu_n\left( [1,2] \right)^2 \right] \PP\left\{ \mu_n\left( [1,2] \right)  > 0  \right\},\] and thus $\PP\left\{ \mu_n\left( [1,2] \right) >0   \right\} \ge c_1^{-1} c_0^2 > 0$, uniformly in $n$. Consequently, by this uniformity, the specified event of positive probability exists.

We now use Theorem~\ref{thm:correlation} to derive the upper bound on the second moment of $\mu_n([1,2])$. Consider first the diagonal terms. We have
\[
\sum_{j=1}^{e^n-1} \E \left [\mu_{j,n}\left( [1,2] \right)^2 \right]  = e^{2\nu n - 2n} \sum_{j=1}^{e^n-1} \E \left[h_{\tau_n}'(x_j)^{2\lambda} I_n(x_j) \right] \le \psi(e^n)  e^{2\nu n - 2n}e^{n - (\lambda \beta + \nu) n } \le \psi(e^n) e^{- d n},
\]
where $d = 1-(\nu-\lambda \beta)  \in (0,1)$.
As for the other terms we have (and write $x_j = x_{j,n}$ etc.)
\begin{align*}
e^{2\nu n - 2n} \sum_{j \neq k} \E \left[h_{\tau_n}'(x_j)^{\lambda}h_{\sigma_n}'(x_k)^\lambda I_n(x_j) I_n(x_k)  \right] &
\lesssim  e^{-2n} \sum_k \sum_{j>k} \psi \left(\frac{e^n}{j-k}\right)\left(\frac{j-k}{e^{n}}\right)^{d-1} \le c.
\end{align*}
In the last step we estimated using the integral $\int_0^1 \psi(1/x) x^{d-1} dx < \infty$ which is finite since $\psi$ is a subpower function and $d-1 > -1$.

Let $\mathcal{E}_\alpha(\mu) = \iint |x-y|^{-\alpha} \mu(dx) \mu(dy)$ be the $\alpha$-dimensional energy of the measure $\mu$. We claim that $\alpha < d$ implies that there is a constant $c<\infty$ depending only on $\alpha$ such that for all $n$,
$\E \left[ \mathcal{E}_\alpha(\mu_n) \right] \le c.$
Indeed, let $\ee> 0$ and $\alpha = d-\ee$. Theorem~\ref{thm:correlation} implies that the contribution from the $O(e^n)$ diagonal terms is $o(1)$:
\begin{align*}
O(e^n) \E \left[ \iint \frac{\mu_{j,n}(dx) \mu_{j,n}(dy)}{|x-y|^{d-\ee}}\right] & \le O(e^n) \psi(e^n) e^{-(2-d+\ee)n}e^{-n(d-1)}=\psi(e^n)O(e^{-\ee n}).
\end{align*}
Moreover, also using Theorem~\ref{thm:correlation}, if $j > k$,
\begin{align*}
\E \left[ \iint \frac{\mu_{j,n}(dx) \mu_{k,n}(dy)}{|x-y|^{d-\ee}}\right] & \le e^{-2n} \psi\left( \frac{e^n}{j-k}\right) \left(\frac{j-k}{e^n}\right)^{-d +\ee + d-1}.
\end{align*}
We sum this first over $j > k$ and then over $k$. The result is bounded by a constant since $\ee > 0$ and $\psi$ is a subpower function.
The estimates hold uniformly in $n$ and we conclude that for every $\ee > 0$ the limiting measure $\mu$ is non-trivial and has finite $(d-\ee)$-dimensional energy on an event of strictly positive probability.

Since the Hausdorff dimension is constant almost surely, a lower bound with positive probability is an an almost sure lower bound. In conclusion we have proved the following theorem, and this concludes the proof of Theorem~\ref{thm:main}.
\begin{theorem}
Let $\beta \in [\beta_-, \hat{\beta}]$. For every $\ee > 0$, \[\PP\left\{ \dim_H V_\beta \ge d-\ee \right\} = 1.\]
\end{theorem}

\bibliographystyle{plain}

\end{document}